\documentclass[11pt,reqno]{amsart}
\usepackage[left=0.9in,top=0.9in,right=0.9in,bottom=0.9in]{geometry}
\usepackage{amsaddr}

\usepackage{amsfonts,amsmath,amsthm,amssymb,latexsym,mathrsfs,stmaryrd}
\usepackage{hyperref}
\usepackage{mathtools}
\usepackage{graphicx}
\usepackage{subcaption}
\usepackage{cleveref}

\usepackage{tikz}\usetikzlibrary{cd,fit,matrix,arrows,decorations.pathmorphing}
\tikzset{commutative diagrams/.cd}

\numberwithin{equation}{section}

\newtheorem{theorem}{Theorem}[section]
\newtheorem{corollary}[theorem]{Corollary}
\newtheorem{lemma}[theorem]{Lemma}
\newtheorem{proposition}[theorem]{Proposition}
\newtheorem{conjecture}[theorem]{Conjecture}
\newtheorem{num-conjecture}[theorem]{Numerically Verified Conjecture}

\theoremstyle{definition}
\newtheorem{definition}[theorem]{Definition}
\newtheorem{definition-theorem}[theorem]{Definition-Theorem}

\newtheorem{fact}[theorem]{Fact}

\theoremstyle{remark}

\newtheorem*{remark*}{Remark}

\usepackage{enumitem}
\setenumerate[1]{label=\textup{(\alph*)}}
\setenumerate[2]{label=\textup{(\roman*)}}

\newcommand\Z{{\mathbb Z}}
\newcommand\N{{\mathbb N}}

\newcommand\C{{\mathbb C}}

\newcommand\F{{\mathbb F}}

\newcommand\Fq{{\mathbb F}_q}

\def\O{\mathcal O}

\DeclareMathOperator{\Hom}{Hom}

\DeclareMathOperator{\Aut}{Aut}

\newcommand\subeq{\subseteq}

\newcommand\bbar{\overline} 
\newcommand\tl{\widetilde} 
\newcommand\hhat{\widehat} 
\newcommand\onto{\twoheadrightarrow}
\newcommand\incl{\hookrightarrow}

\DeclarePairedDelimiter{\abs}{\lvert}{\rvert}

\DeclarePairedDelimiter{\set}{\{}{\}}

\DeclarePairedDelimiter\ceil{\lceil}{\rceil}
\DeclarePairedDelimiter\floor{\lfloor}{\rfloor}

\newcommand{\defn}{\textbf}

\newcommand{\qbinom}[2]{{#1\brack #2}}
\newcommand{\m}{\mathfrak{m}}
\newcommand{\zetahat}{\hhat{\zeta}}
\newcommand{\nuhat}{\hhat{\nu}}
\newcommand{\fc}{\mathfrak{c}}
\newcommand{\Fqsq}{\F_{q^2}}

\newcommand{\Sur}{\mathrm{Sur}}
\newcommand{\Quot}{\mathrm{Quot}}

\newcommand{\type}{\mathrm{type}}
\newcommand{\Dir}{\mathrm{Dir}}

\newcommand{\down}{\downarrow}

\newcommand{\AG}{\mathbf{AG}}
\newcommand{\Br}{\mathbf{Br}}

\begin{document}
\title{Coh zeta functions for inert quadratic orders}

\author{Yifeng Huang}
\address{Dept.\ of Mathematics, University of Southern California}
\email{yifenghu@usc.edu}

\date{\today}

\begin{abstract}
    We study the Coh zeta function for a family of inert quadratic orders, which we conjecture to be given by $t$-deformed Bressoud $q$-series. This completes a trilogy connecting the zeta functions of ramified and split quadratic orders to the classical Andrews--Gordon and Bressoud identities, respectively \cite{huangjiang2023torsionfree, shane2024multiple}. We provide strong evidence for this conjecture by deriving the first explicit formulas for the finitized Coh zeta function of the simplest order in the family, and for the $t=1$ specialization of the finitized Coh zeta functions for all orders in the family. Our primary tool is a new method based on M\"obius inversion on posets.
\end{abstract}
\maketitle
\setcounter{tocdepth}{1}
\tableofcontents

\section{Introduction}

Let $R$ be a commutative ring with suitable finiteness conditions (e.g., finitely generated over $\Z$ or a completion thereof). We define the \defn{Coh zeta function} by the formal Dirichlet series
\begin{equation}
    \zetahat_R(s):=\sum_{Q} \frac{1}{\abs{\Aut_R(Q)}}\abs{Q}^{-s},
\end{equation}
where the sum is over all isomorphism classes of finite-cardinality $R$-modules $Q$; cf.~\cite{huang2023mutually}. This appears as a zeta function considered by Cohen and Lenstra in their seminal work \cite{cohenlenstra1984heuristics} in arithmetic statistics, and it is directly connected to the stack of coherent sheaves, the counting of matrix points on varieties, and the (unframed version of) degree 0 Donaldson--Thomas theory \cite{bbs2013motivic, bryanmorrison2015motivic, hos2023matrix}.

When $R$ is the coordinate ring of a singular curve over a finite field $\Fq$, the study of $\zetahat_R(s)$ has surprisingly led to new sources of $q$-series, identities, and modular functions \cite{huang2023mutually, huangjiang2023torsionfree}. A central open problem is to understand the nature of these $q$-series and to develop a framework for predicting their structure, potentially drawing connections to fields like knot theory or mathematical physics (for a strong hint, see \cite{ors2018homfly}). However, progress has been challenging. Explicit computation of $\zetahat_R(s)$ is notoriously difficult, with only a couple of families beyond the initial case \cite{huang2023mutually} having been computed \cite{huangjiang2023torsionfree}. Furthermore, the initial formulas for these families were often complicated multi-sums that did not immediately reveal their key properties. While the two-variable series (in $q$ and $t=q^{-s}$) in the examples known so far did not appear to have product forms, their specializations at $t = \pm 1$ do, which strongly suggested the existence of a richer, more elegant structure beyond the Rogers--Ramanujan type identity observed in the specializations \cite{huangjiang2023torsionfree}.

The author's joint work with Jiang \cite{huangjiang2023torsionfree} studied two specific families of non-maximal orders in quadratic extensions of $\Fq((X))$. The first is the family of \textbf{ramified} orders $R_{2,2m+1} := \Fq[[X,Y]]/(Y^2-X^{2m+1})$, whose geometry corresponds to the $(2,2m+1)$ torus knot. The second is the family of \textbf{split} orders $R_{2,2m}$, which correspond to the $(2,2m)$ torus link; these rings are defined as $\Fq[[X,Y]]/(Y(Y-X^m))$, which for odd $q$ are isomorphic to $\Fq[[X,Y]]/(Y^2-X^{2m})$; cf.~\cite{huangjiang2023torsionfree}. These two families (with $m\geq 1$) completely classify the non-maximal orders in the ramified quadratic extension $\Fq((\sqrt{X}))$ and the split quadratic extension $\Fq((X))\times \Fq((X))$ of $\Fq((X))$, respectively; cf.~\cite{saikia1988}. The Coh zeta function of $R_{2,2m+1}$ was identified as a $t$-deformed version of the $m$-fold sum in one of the Andrews--Gordon identities, while that of $R_{2,2m}$ was a mysterious $2m$-fold sum at that time.

A significant simplification for the split case was recently achieved by Shane Chern \cite{shane2024multiple}, in response to questions communicated by the author in June 2024 during the \emph{Legacy of Ramanujan} conference. This work was pivotal, as it allows the full picture for the ramified and split cases to be presented in the following unified and elegant form. To state this precisely, we first recall the \textbf{lattice zeta function} (or \textbf{Quot zeta function}) for an $R$-module $M$:
\[\zeta_M^R(s) := \sum_{L \subseteq_R M} (M:L)^{-s},\]
where the sum is over all finite-index submodules $L$ of $M$. With this, we define the \textbf{finitized Coh zeta function} as $\zetahat_{R,n}(s) := \zeta_{R^n}^R(s+n)$; cf.~\Cref{thm:framing-inf}. Define also the following polynomials, whose $t=1$ specializations appear as finitized Andrews--Gordon and Bressoud multi-sums in \cite{paule1985,asw1999}:
\[ \AG_n(q,t;2m+3):=(q;q)_n \sum_{n_1,\dots,n_m} \frac{q^{\sum n_i^2}t^{2\sum n_i}}{(q;q)_{n-n_1}(q;q)_{n_1-n_2}\cdots (q;q)_{n_m}}, \]
\[ \Br_n(q,t;2m+2):=(q;q)_n \sum_{n_1,\dots,n_m} \frac{q^{\sum n_i^2} t^{2\sum n_i}}{(q;q)_{n-n_1}(q;q)_{n_1-n_2}\cdots (q;q)_{n_m}(-tq;q)_{n_m}}. \]
Throughout the paper, we use the change of variables $t=q^{-s}$.

\begin{theorem}[{\cite{huangjiang2023torsionfree, shane2024multiple}}] \label{thm:known}
For each $m \ge 1$, the finitized Coh zeta functions for the ramified and split orders are given by:
\begin{align*}
    \zetahat_{R_{2,2m+1},n}(s) &= \frac{1}{(tq^{-1};q^{-1})_n} \AG_n(q^{-1},t;2m+3) \\
    \zetahat_{R_{2,2m},n}(s) &= \frac{1}{(tq^{-1};q^{-1})_n} \Br_n(q^{-1},-t;2m+2).
\end{align*}
\end{theorem}

This newly clarified picture naturally raises two questions. First, what geometric object corresponds to the more direct $t$-deformation, $\Br_n(q,t;2m+2)$? Second, the formula for the ramified case is a polynomial in $t^2$, making $\AG_n(q,t;2m+3)=\AG_n(q,-t;2m+3)$. Is there a structural explanation for this symmetry?

In this paper, we propose that the answers to both questions lie with the third and final type of quadratic extension: the \textbf{inert} case. We consider the family of inert quadratic orders $R'_{2,2m} := \Fq[[T]]+T^{m}\F_{q^{2}}[[T]]$, which classify the non-maximal orders of the inert extension $\Fqsq((T))/\Fq((T))$. Our central conjecture provides the missing piece to the puzzle.

\begin{conjecture}\label{conj:main}
For each $m \ge 1$, the finitized Coh zeta function of the inert quadratic order $R'_{2,2m}$ is given by the $t$-deformed Bressoud sum:
\[ \zetahat_{R'_{2,2m},n}(s) = \frac{1}{(tq^{-1};q^{-1})_n} \Br_n(q^{-1},t;2m+2). \]
\end{conjecture}

A key structural insight, which motivates our work, is that these three classes of orders are related by \textbf{quadratic twists}. For odd $q$, let $\alpha \in \Fq^\times$ be a non-square. The inert order $R'_{2,2m}$ is isomorphic to the ring $\Fq[[X,Y]]/(Y^2 - \alpha X^{2m})$, making it a quadratic twist of the split order $R_{2,2m} \simeq \Fq[[X,Y]]/(Y^2-X^{2m})$. At the same time, we can define the quadratic twist of the ramified order as $R'_{2,2m+1}:=\Fq[[X,Y]]/(Y^2-\alpha X^{2m+1})$, which turns out to be isomorphic to $R_{2,2m+1}$ itself via a simple change of variables. This isomorphism provides a powerful structural explanation for the second question: since the rings are the same, their zeta functions must be identical, which in the context of the proposed framework forces the identity $\AG_n(q^{-1},-t;2m+3) = \AG_n(q^{-1},t;2m+3)$ and thus explains the vanishing of the odd powers of $t$.

This suggests a complete and uniform picture where \Cref{conj:main} and \Cref{thm:known} can be listed together:
\begin{align*}
    \zetahat_{R_{2,2m+1},\,n}(s) &= \frac{1}{(tq^{-1};q^{-1})_n} \AG_n(q^{-1},-t;2m+3),& \zetahat_{R'_{2,2m+1},\,n}(s) &= \frac{1}{(tq^{-1};q^{-1})_n} \AG_n(q^{-1},t;2m+3),\\
    \zetahat_{R_{2,2m},\,n}(s) &= \frac{1}{(tq^{-1};q^{-1})_n} \Br_n(q^{-1},-t;2m+2),& \zetahat_{R'_{2,2m},\,n}(s) &= \frac{1}{(tq^{-1};q^{-1})_n} \Br_n(q^{-1},t;2m+2). 
\end{align*}

The primary contribution of this paper is to provide the first explicit formulas for the finitized Coh zeta function of the inert orders, $\zetahat_{R'_{2,2m},n}(s)$, in two crucial cases: the case $m=1$ (for all $s$) and the $s=0$ (i.e., $t=1$) specialization for all $m\geq 1$ (see \Cref{thm:m=1-main-result} and \Cref{thm:s-zero-evals}, respectively). We have numerically verified that our new formulas match the $t$-deformed Bressoud sums predicted in \Cref{conj:main}. This paper does not contain a proof of the corresponding $q$-series identities; a proof will appear in forthcoming joint work with Chern.

Our results are achieved by introducing a new computational technique based on \emph{M\"obius inversion on the poset of submodules} (see \Cref{sec:saturation}). This approach is different from prior work \cite{huangjiang2023torsionfree} and allows us to handle the inert case, where previous methods faced technical obstacles. As a byproduct of this new technique, we can also re-derive some formulas for the already-solved ramified and split cases (see \Cref{thm:s-zero-evals}(a)(b)). Comparing our new formulas with the known ones yields several $q$-series identities, providing them with proofs that originate from the module-counting framework rather than from traditional $q$-series manipulations.

\subsection*{Plan of the paper}
In \Cref{sec:prelim}, we collect the necessary background on $q$-series and results about lattice zeta functions from \cite{huangjiang2023torsionfree}. In \Cref{sec:saturation}, we introduce our main computational tool, M\"obius inversion on posets, to prove the key technical result, \Cref{thm:rtilde}. This result is then applied in \Cref{sec:bounded} to compute the $s=0$ specializations for all three families of orders, as well as in \Cref{sec:m=1} to compute the finitized Coh zeta function at all $s$ for the $m=1$ inert case.

\section{Preliminaries}\label{sec:prelim}

\subsection{Basic hypergeometric series}
We begin by defining the standard notation used in basic hypergeometric series. The \textbf{$q$-Pochhammer symbol} is defined for $n \in \N \cup \{\infty\}$ by
\[ (a;q)_n := \prod_{k=0}^{n-1} (1-aq^k). \]
For brevity, we often write $(a_1, \dots, a_r; q)_n := (a_1;q)_n \cdots (a_r;q)_n$. The \textbf{$q$-binomial coefficient} is defined for integers $n,k$ by
\[ \qbinom{n}{k}_q := \frac{(q;q)_n}{(q;q)_k (q;q)_{n-k}}. \]
We use the standard convention $1/(q;q)_n=0$ if $n\in \Z_{<0}$, which implies that sums over integers are naturally truncated. For example, $\sum_{r\in \Z} \frac{1}{(q;q)_{n-r}(q;q)_{n+r}} = \sum_{r=-n}^n \frac{1}{(q;q)_{n-r}(q;q)_{n+r}}$.

The finitizations of the multi-sums from the introduction can be expressed as single sums.
\begin{fact}[{\cite{paule1985, asw1999}}]
For $m \ge 1$ and $n \in \N$, we have the following identities:
\begin{align}
    \AG_n(q,1;2m+3) &= (q;q)_n \sum_{r} \frac{(-1)^r q^{\binom{r}{2}+(m+1)r^2}}{(q;q)_{n-r}(q;q)_{n+r}}, \label{eq:finitized-ag} \\
    \Br_n(q,1;2m+2) &= (q;q)_n \sum_{r} \frac{(-1)^r q^{(m+1)r^2}}{(q;q)_{n-r}(q;q)_{n+r}}. \label{eq:finitized-br}
\end{align}
\end{fact}
These limits imply the corresponding Rogers--Ramanujan type identities by taking $n\to \infty$ and applying the Jacobi triple product formula. 

\begin{corollary}
The $n\to\infty$ limit of \eqref{eq:finitized-ag} is the central Andrews--Gordon identity \cite[Corollary~7.8]{andrewspartitions}:
\begin{equation}
    \AG_\infty(q,1;2m+3) := \sum_{n_1,\dots,n_m} \frac{q^{\sum_{i=1}^m n_i^2}}{(q;q)_{n_1-n_2}\cdots(q;q)_{n_{m-1}-n_m}(q;q)_{n_m}} = \frac{(q^{m+1},q^{m+2},q^{2m+3};q^{2m+3})_\infty}{(q;q)_\infty}.
\end{equation}
The $n\to\infty$ limit of \eqref{eq:finitized-br} is the central Bressoud identity \cite{bressoud1980}:
\begin{equation}
    \Br_\infty(q,1;2m+2) := \sum_{n_1,\dots,n_m} \frac{q^{\sum_{i=1}^m n_i^2}}{(q;q)_{n_1-n_2}\cdots(q;q)_{n_{m-1}-n_m}(q^2;q^2)_{n_m}} = \frac{(q^{m+1},q^{m+1},q^{2m+2};q^{2m+2})_\infty}{(q;q)_\infty}.
\end{equation}
\end{corollary}
We call them the ``central'' identities because each is part of a larger family; for instance, the general Andrews--Gordon identities include product sides of the form $(q^{a},q^{2m+3-a},q^{2m+3};q^{2m+3})_\infty/(q;q)_\infty$ for other values of $a$.

\subsection{Formal Dirichlet series}
For any commutative ring $A$ with $1$, the ring of \defn{formal Dirichlet series} with coefficient ring $A$, denoted by $\Dir(A)$, is a ring whose additive group is
\begin{equation}
    \Dir(A)=A^\N:=\set{(a_1,a_2,\dots):a_n\in A}, 
\end{equation} 
and whose multiplication is given by
\begin{equation}
    (a_n)_n \cdot (b_n)_n:=(\sum_{d|n} a_d b_{n/d})_n.
\end{equation}
We identify an element $(a_n)_n\in \Dir(A)$ with a formal series $\sum_{n=1}^\infty a_n n^{-s}$ and write $\Dir(A)=\Dir(A;s)=\set{\sum_{n=1}^\infty a_n n^{-s}: a_n\in A}.$ We call $a_1$ the constant term. Unless otherwise specified, we view every Dirichlet series formally, without assuming convergence. Given a sequence $f, f_1,f_2,\dots\in \Dir(\C;s)$, we say $f_n(s)$ converges coefficient-wise to $f(s)$ as $n\to \infty$ if for each $N\geq 1$, the $N$-th coefficient of $f_n(s)$ converges to the $N$-th coefficient of $f(s)$. 

\subsection{Quot zeta functions}
For a commutative ring $R$ and an $R$-module $M$ with a finite number of submodules of any given index, the \defn{Quot zeta function} of $M$ over $R$ is the formal Dirichlet series
\begin{equation}
    \zeta_M(s)=\zeta_M^R(s):=\sum_{L\subeq_R M} (M:L)^{-s},
\end{equation}
where $L$ ranges over finite-index $R$-submodules of $M$. We also recall the \defn{Coh zeta function}
\begin{equation}
    \zetahat_R(s):=\sum_{Q/R} \abs{\Aut(Q)}^{-1}\abs{Q}^{-s},
\end{equation}
where the sum is over all isomorphism classes of finite $R$-modules $Q$. From now on, the free module $A^{\oplus n}$ over any ring $A$ will be simply written as $A^n$.

\begin{theorem}[{cf.~\cite[Proposition 3.15]{huang2024commuting}}] \label{thm:framing-inf}
    Let $R$ be any commutative ring with $1$. The following are equivalent:
    \begin{enumerate}
        \item $\zetahat_R(s)$ is well-defined (i.e., there are finitely many isomorphism classes of $R$-modules of any given finite cardinality).
        \item $\zeta_{R^n}^R(s)$ is well-defined for all $n\geq 0$ (i.e., there are finitely many $R$-submodules of $R^n$ of any given finite index).
    \end{enumerate}
    If these conditions hold, the formal Dirichlet series
    \begin{equation}
        \zetahat_{R,n}(s):=\zeta_{R^n}^R(s+n)
    \end{equation}
    converge coefficient-wise as $n\to \infty$. Moreover, the limit is
    \begin{equation}\label{eq:framing-inf-limit}
        \zetahat_R(s)=\lim_{n\to\infty} \zetahat_{R,n}(s).
    \end{equation}
\end{theorem}
This theorem justifies calling $\zetahat_{R,n}(s)$ the \textbf{finitized Coh zeta function}.

\begin{proof}
    For $n\geq 0$ and $Q$ a finite-cardinality module over $R$, let $\Sur_R(R^n,Q)$ be the set of $R$-linear surjections from $R^n$ to $Q$. The group $\Aut(Q)$ acts freely on $\Sur_R(R^n,Q)$ on the left by composition. Since every surjection $R^n\onto Q$ gives rise to an exact sequence $0\to L\to R^n\to Q\to 0$, we have a canonical bijection
    \begin{equation}
        \Sur_R(R^n,Q)/\Aut(Q)\overset{\simeq}{\to}\Quot_R(R^n,Q):=\set{L\subeq_R R^n: R^n/L\simeq Q}.
    \end{equation}
    It follows that $\abs{\Quot_R(R^n,Q)}=\abs{\Sur_R(R^n,Q)}/\abs{\Aut(Q)}$. The equivalence of the conditions (a) and (b) is now straightforward.

    To prove \eqref{eq:framing-inf-limit}, the key observation is that for any finite $R$-module $Q$,
    \begin{equation}
        \lim_{n\to \infty} \frac{\abs{\Sur_R(R^n,Q)}}{\abs{\Hom_R(R^n,Q)}}=1,
    \end{equation}
    where $\Hom_R(R^n,Q) \simeq Q^n$. This limit follows from a simple inclusion-exclusion argument. The $N$-th Dirichlet coefficient of $\zetahat_{R,n}(s)$ is $\sum_{\abs{Q}=N} N^{-n}\abs{\Quot_R(R^n,Q)}$. Taking the limit as $n\to\infty$, we get
    \begin{align*}
        \lim_{n\to \infty} \sum_{\abs{Q}=N} N^{-n}\frac{\abs{\Sur_R(R^n,Q)}}{\abs{\Aut(Q)}} 
        &= \sum_{\abs{Q}=N} \frac{1}{\abs{\Aut(Q)}} \lim_{n\to \infty} \frac{\abs{\Sur_R(R^n,Q)}}{\abs{Q}^n} \\
        &= \sum_{\abs{Q}=N} \frac{1}{\abs{\Aut(Q)}},
    \end{align*}
    which is the $N$-th coefficient of $\zetahat_R(s)$, as required.
\end{proof}

If $R$ is a local ring with maximal ideal $\m$, the Quot zeta functions of $R^n$ and of the submodule $\m R^n$ are related as follows.
\begin{lemma}[{\cite[Corollary 2.7]{huangjiang2023punctual}}] \label{lem:nakayama}
    If $R$ is a local ring with maximal ideal $\m$ and residue field $\Fq$, then for any $n\in \Z_{\geq 0}$,
    \begin{equation}
        \zeta_{R^n}^R(s)=\sum_{r=0}^n \qbinom{n}{r}_q q^{-rs}\zeta_{\m R^r}^R(s-n+r).
    \end{equation}
\end{lemma}

\subsection{Lattices}
Let $F$ be a local non-archimedean field with ring of integers $\O_F$ and residue field $k=\Fq$. Let $E$ be a finite product of finite extensions of $F$. An \defn{order} over $\O_F$ in $E$ is an $\O_F$-subalgebra $R\subeq E$ that is finitely generated as an $\O_F$-module and spans $E$ over $F$. A (full) $R$-\defn{lattice} in $E^n$ is a finitely generated $R$-submodule $L\subeq E^n$ that spans $E^n$ over $F$. An order $R$ over $\O_F$ is \defn{monogenic} if $R=\O_F[\alpha]$ for some $\alpha\in R$.

Let $\tl R$ be the normalization of $R$. The \defn{conductor} of $R$ is the largest ideal of $\tl R$ contained in $R$, denoted by $\fc_R$.
\begin{lemma}[{\cite{huangjiang2023torsionfree}}] \label{lem:tlRn}
    Let $R$ be an order. For the inclusion of quotient rings $A=R/\fc_R\subeq B=\tl R/\fc_R$, we have
    \begin{equation}\label{eq:tlRn-formula}
        \zeta_{\tl R^n}^R(s)=\zeta_{\tl R^n}^{\tl R}(s) \varepsilon_{B^n}^{A\subeq B}(s),
    \end{equation} 
    where $\varepsilon_{M}^{A\subeq B}(s):=\sum_{L\subeq_A M,\; BL=M} (M:L)^{-s}$ is the \textbf{saturation zeta function}.
\end{lemma}
\begin{lemma}[{\cite{solomon1977zeta}}] \label{lem:solomon}
    If $\tl R=\prod_{i=1}^l \tl R_i$, where each $\tl R_i$ is a DVR with residue field $\F_{q_i}$, then
    \begin{equation}\label{eq:solomon-formula}
        \zeta_{\tl R^n}^{\tl R}(s)=\prod_{i=1}^l (q_i^{-s};q_i)_n^{-1},
    \end{equation}
    and equivalently,
    \begin{equation}
        \zetahat_{\tl R,n}(s)=\zeta_{\tl R^n}^{\tl R}(s+n)=\prod_{i=1}^l (q_i^{-1-s};q_i^{-1})_n^{-1}.
    \end{equation}
\end{lemma}

\subsection{Reflection principle} 
The \defn{normalized lattice zeta function} for a lattice $M \subseteq E^n$ is $\nu_M^R(s):=\zeta_M^R(s)/\zeta_{\tl R^n}^{\tl R}(s)$. This is known to be a Dirichlet polynomial \cite[Proposition 4.7]{huangjiang2023torsionfree}, and hence entire in $s$. Correspondingly, the \defn{normalized finitized Coh zeta function} is
\begin{equation}
    \nuhat_{R,n}(s) := \nu_{R^n}^R(s+n) = \frac{\zetahat_{R,n}(s)}{\zetahat_{\tl R,n}(s)}.
\end{equation}
These normalized functions satisfy two crucial properties.
\begin{lemma}[{\cite[Remark 4.10]{huangjiang2023torsionfree}}] \label{lem:tlR-same-0}
    For any order $R$, $\nu_{R^n}^R(0)=\nu_{\tl R^n}^R(0)$.
\end{lemma}
An order is \defn{Gorenstein} if its dualizing module is projective. The quadratic orders $R_{2,2m+1}, R_{2,2m}$, and $R'_{2,2m}$ studied in this paper are all monogenic, and thus Gorenstein \cite[Prop.~3.6]{jensenthorup2015}.
\begin{theorem}[Reflection Principle {\cite[Theorem 1.7]{huangjiang2023torsionfree}}] \label{thm:reflection}
    If $R$ is a Gorenstein order and $D=\abs{\tl R/R}$, then $\nu_{R^n}^R(s)=D^{n^2-2ns} \nu_{R^n}^R(n-s)$.
\end{theorem}
Rewriting these results in terms of $\nuhat_{R,n}(s)$, the reflection principle becomes $$\nuhat_{R,n}(s)=D^{-n^2-2ns} \nuhat_{R,n}(-s-n),$$ and \Cref{lem:tlR-same-0} implies $\nuhat_{R,n}(-n)=\nu_{\tl R^n}^R(0)$. Combining these gives a fundamental formula for the $s=0$ evaluation:
\begin{equation} \label{eq:nuhat-at-zero}
    \nuhat_{R,n}(0)=D^{-n^2}\nu_{\tl R^n}^R(0).
\end{equation}
This relation explains why the $s=0$ specialization is computationally special, as $\nuhat_{R,n}(0)$ can be computed from the zeta function of the simpler module $\tl R^n$.

\section{Saturation zeta function and the key technical theorem}\label{sec:saturation}

In this section, we prove our main technical result, \Cref{thm:rtilde}, which provides explicit formulas for the lattice zeta function $\zeta_{\tl R^n}^R(s)$ for the three families of quadratic orders. Our primary new tool is M\"obius inversion on posets of modules. We begin by recalling the combinatorial definitions needed to state the theorem.

For two tuples of integers $\mathbf{r}=(r_1,\dots,r_m)$ and $\mathbf{s}=(s_1,\dots,s_m)$, we define a version of the Hall polynomial as follows.
\begin{definition}[cf.~\Cref{thm:hall_skew}]
Let $\mathbf{r}=(r_1,\dots,r_m)$ and $\mathbf{s}=(s_1,\dots,s_m)$ be tuples of integers. We define
\begin{equation}\label{eq:hall_skew_hypergeom}
    G^{\mathbf{r}}_{\mathbf{s}}(q):=q^{\sum_{i=1}^m s_i(r_i-s_i)}\prod_{i=1}^m \qbinom{r_i-s_{i+1}}{r_i-s_i}_{q^{-1}} \in \Z[q],
\end{equation}
where $r_{m+1}=s_{m+1}=0$ by convention. Note that $G^{\mathbf{r}}_{\mathbf{s}}(q)=0$ unless $r_1\geq \dots\geq r_m\geq 0$, $s_1\geq \dots\geq s_m\geq 0$, and $r_i\geq s_i$ for all $i$. We also use the notation $\abs{\mathbf{s}}:=s_1+\dots+s_m$.
\end{definition}

With this, we can now state the main technical theorem of the paper, which provides explicit formulas for the lattice zeta function of $\tl R^n$ considered as an $R$-module.

\begin{theorem}[Key Theorem]
    \label{thm:rtilde}
    For any $m\in \Z_{\geq 1}, n\in \Z_{\geq 0}$ and any prime power $q$, using the notation $t:=q^{-s}$ and the multi-index notation such as $((2n)^{m-1},2n-r):=(\underbrace{2n,\dots,2n}_{m-1},2n-r)$, the following hold:
    \begin{enumerate}
        \item (Ramified) If $R=R_{2,2m+1} = \Fq[[X,Y]]/(Y^2-X^{2m+1})$, then
        \begin{equation}
            (q^{-s};q)_n \zeta_{\tl R^{\oplus n}}^R(s) = \sum_{r\in \Z,\mathbf{s}\in \Z^m} \qbinom{n}{r}_q (-1)^r q^{\binom{r}{2}} t^{2mn-\abs{\mathbf{s}}} G^{((2n)^{m-1},2n-r)}_{\mathbf{s}}(q).
        \end{equation}
        \item (Split) If $R=R_{2,2m} = \Fq[[X,Y]]/(Y(Y-X^{m}))$, then
        \begin{equation}
            (q^{-s};q)_n^2 \zeta_{\tl R^{\oplus n}}^R(s) = \sum_{r_1,r_2\in \Z, \mathbf{s}\in \Z^m}\qbinom{n}{r_1}_q\qbinom{n}{r_2}_q(-1)^{r_1+r_2} q^{\binom{r_1}{2}+\binom{r_2}{2}}t^{2mn-\abs{\mathbf{s}}}G^{((2n)^{m-1},2n-r_1-r_2)}_\mathbf{s}(q).
        \end{equation}
        \item (Inert) If $R=R'_{2,2m} = \Fq[[T]]+T^m\F_{q^2}[[T]]$, then
        \begin{equation}
            (q^{-2s};q^2)_n \zeta_{\tl R^{\oplus n}}^R(s) = \sum_{r\in \Z, \mathbf{s}\in \Z^m} \qbinom{n}{r}_{q^2} (-1)^r q^{r^2-r} t^{2mn-\abs{\mathbf{s}}}G^{((2n)^{m-1},2n-2r)}_{\mathbf{s}}(q).
        \end{equation}
    \end{enumerate}
\end{theorem}

The remainder of this section will be dedicated to the proof of this theorem.

\subsection{Posets and M\"obius function}
Let $(\mathcal{P},\preceq)$ be a finite poset. Recall (say, \cite[\S 3.7]{stanleyec1}) that the M\"obius function $\mu_{\mathcal{P}}(x,y)$ is the unique function with inputs in $x,y\in \mathcal{P}, x\preceq y$ and values in $\Z$ that satisfies the M\"obius inversion property: For functions $f,g$ on $\mathcal{P}$ that take values in any abelian group, we have
    \begin{equation}
        f(x)=\sum_{y\preceq x} g(y) \text{ for all }x\in \mathcal{P}\iff g(x)=\sum_{y\preceq x} \mu_{\mathcal{P}}(y,x)f(y) \text{ for all }y\in \mathcal{P}.
    \end{equation}
The M\"obius function also satisfies that the value $\mu_{\mathcal{P}}(x,y)$ only depends on the structure of the poset $[x,y]:=\set{z\in \mathcal{P}:x\preceq z\preceq y}$. 

For any finite ring $B$ and a finite module $M$, let $\mathcal{P}_B(M)$ be the poset of $B$-submodules of $M$ ordered by inclusion. For any $W_1\subeq_B W_2\subeq_B M$, let $\mu_B(W_2/W_1)=\mu_B(W_1,W_2)$ denote the M\"obius value with respect to $\mathcal{P}_B(M)$ (the notation $\mu_B(W_2/W_1)$ is justified because the poset $[W_1,W_2]$ is determined by the structure of the quotient $B$-module $W_2/W_1$). 

\begin{lemma}\label{lem:ext-fiber-moebius}
    For any inclusion of finite rings $A\subeq B$ and finite $B$-module $\tl M$, we have
    \begin{equation}\label{eq:ext-fiber-moebius}
        \varepsilon_{\tl M}^{A\subeq B}(s)=\sum_{\tl W\subeq_B \tl M} \mu_B(\tl M/\tl W) (\tl M:\tl W)^{-s} \zeta_{\tl W}^A(s).
    \end{equation}
\end{lemma}
\begin{proof}
    On the poset $\mathcal{P}_B(\tl M)$, construct the functions that take values in Dirichlet polynomials
    \begin{equation}
        f(\tl W)= \sum_{W\subeq_A \tl W} (\tl M:W)^{-s},\; g(\tl W)=\sum_{W\subeq_A \tl W,\; BW=\tl W} (\tl M:W)^{-s},\; \tl W\in \mathcal{P}_B(\tl M).
    \end{equation}
    It is trivial to verify that $f(\tl W)=\sum_{\tl W'\subeq_B \tl W} g(\tl W')$, $f(\tl W)=(\tl M:\tl W)^{-s} \zeta_{\tl W}^A(s)$, and $g(\tl W)=(\tl M:\tl W)^{-s}\varepsilon_{\tl W}^{A\subeq B}(s)$. The desired identity follows from M\"obius inversion and specializing $\tl W=\tl M$.
\end{proof}

\subsection{The poset of DVR modules}
We recall an important setting where the poset $\mathcal{P}_B(M)$ can be understood fully combinatorially and explicitly. Let $(V,\pi,\Fq)$ be an DVR with uniformizer $\pi$ and finite field $\Fq$. Given a finite module $M$ over $V$, the \defn{type} of $M$ over $V$ is the unique partition $\lambda=(\lambda_1,\lambda_2,\dots)$ such that $M\simeq_V M_V(\lambda):=\bigoplus_i V/\pi^{\lambda_i}$. 

Let $B=V/\pi^m$ for some $m\geq 1$. Then a $B$-module $M$ is simply a $V$-module whose type $\lambda$ satisfies $\lambda_1\leq m$. In the world of $B$-modules, $\simeq_B$ and $\simeq_V$ are equivalent, and $\subeq_B$ and $\subeq_V$ mean the same thing. For a finite $B$-module $M$, we refer to the type of $M$ over $V$ as the type of $M$ over $B$, and denote $\type_B(M)=\type_V(M)$. Conversely, for $\lambda_1\leq m$, denote $M_B(\lambda)=M_V(\lambda)$. For $V$-modules $L\subeq M$, the \defn{cotype} of $L$ in $M$ is the type of $M/L$. 

The following two fundamental results are due to Philip Hall.

\begin{theorem}[{\cite[\S II]{macdonaldsymmetric}}]
    For any partitions $\lambda,\mu,\nu$, there is a universal polynomial (called the \defn{Hall polynomial}) $g^\lambda_{\mu\nu}(t)\in \Z[t]$ such that for any DVR $(V,\pi,\Fq)$, the number of $V$-submodules of $M_V(\lambda)$ of type $\mu$ and cotype $\nu$ is given by $g^{\lambda}_{\mu\nu}(q)$. Moreoever, $g^{\lambda}_{\mu\nu}(t)=g^\lambda_{\nu\mu}(t)$.
\end{theorem}

\begin{lemma}[{\cite[3.9.5, 3.10.2]{stanleyec1}}]
    Let $L\subeq M$ be finite modules over $(V,\pi,\Fq)$. Then
    \begin{equation}
        \mu_V(M/L)=\begin{cases}
            (-1)^r q^{\binom{r}{2}}, & \type_V(M/L)=(1^r);\\
            0, & \text{otherwise,}
        \end{cases}
    \end{equation}
    where as usual, $(1^r)$ denotes the partition consisting only of $r$ $1$'s.
    \label{lem:moebius-vanishing}
\end{lemma}

We will only use a coarser version of the Hall polynomial.
\begin{definition}
    For partitions $\lambda,\mu$, let $g^\lambda_\mu(t):=\sum_\nu g^\lambda_{\mu\nu}(t)$. 
\end{definition} 
Clearly, $g^\lambda_\mu(q)$ counts submodules of $M_V(\lambda)$ of type $\mu$, as well as submodules of $M_V(\lambda)$ of cotype $\mu$. An explicit expression of $g^\lambda_\mu(t)$ is as follows:

\begin{theorem}[{\cite{warnaar2013remarks}}]
    For partitions $\lambda,\mu$, we have
    \begin{equation}
        g^{\lambda}_{\mu}(t)=t^{\sum_{i\geq 1}\mu_i'(\lambda_i'-\mu_i')}\prod_{i\geq 1} {\lambda_i'-\mu_{i+1}' \brack \lambda_i'-\mu_i'}_{t^{-1}},
    \end{equation}
    where as usual, $\lambda_i'$ denotes the $i$-th part of the conjugate partition $\lambda'$ of $\lambda$.
\label{thm:hall_skew}
\end{theorem}

\begin{remark*}
    For partitions $\lambda,\mu$, say $\lambda_1,\mu_1\leq m$, then if $\mathbf{r}=(r_1,\dots,r_m)$ and $\mathbf{s}=(s_1,\dots,s_m)$ are defined by $r_i=\lambda_i'$, $s_i=\mu_i'$, then $g^{\lambda}_{\mu}(t)=G^{\mathbf{r}}_{\mathbf{s}}(t)$ in \eqref{eq:hall_skew_hypergeom}. In other words, $g^\lambda_\mu(t)=G^{\lambda'}_{\mu'}(t)$.
\end{remark*}

We also need a well-known fact about rectangular partitions. 
\begin{lemma}
    Let $(m^n)$ denote the partition $(m,\dots,m)$ with $n$ parts of $m$. Then for any $\mu\subeq (m^n)$, there is a unique partition $\nu$ such that $g^{(m^n)}_{\mu\nu}\neq 0$. Moreover, $\nu$ is the complement of $\mu$ in $(m^n)$ (so $\nu_i=m-\mu_{n+1-i}$ for $1\leq i\leq n$ and $\nu_{n+1}=0$), denoted by $(m^n)-\mu$. 
    \label{lem:rectangular}
\end{lemma}

We will often expand the Hall polynomial of the following simple form without mentioning:
\begin{equation}
    g^{(m^n)}_{(1^r)}(q)=G^{(n^m)}_{(r)}(q)=\qbinom{n}{r}_q, \quad \text{if $m\geq 1$.}
\end{equation}

\subsection{The key lemma}
\begin{definition}
    A \defn{finite DVR quotient} is a ring of the form $A=V/\pi^m$ for some DVR $(V,\pi,\Fq)$ and $m\geq 1$. We denote $q_A:=q$ and $m_A:=m$ as they are determined by the ring $A$. For $A\subeq B$ an inclusion of finite DVR quotients and any partition $\lambda$ with $\lambda_1\leq m_B$, we define $\lambda\down_A^B$ to be the type of $M_B(\lambda)$ viewed as an $A$-module.
\end{definition}
In all relevant examples of $A\subeq B$, $\lambda\down_A^B$ will have an easy description in terms of $\lambda$. Note in general that $q_B^{\abs{\lambda}}=q_A^{\abs{\lambda\down_A^B}}$.

\begin{lemma}[Key lemma]
    Let $A\subeq B$ be an inclusion of finite DVR quotients. Then for any partition $\lambda$ with $\lambda_1\leq m_B$,
    \begin{equation}
        \varepsilon_{M_B(\lambda)}^{A\subeq B}(s)=\sum_{\mu,r} g^{\lambda}_{\mu,(1^r)}(q_B) (-1)^r q_B^{\binom{r}{2}}  \sum_{\rho} g^{\mu\down_A^B}_\rho(q_A) (q_B^{\abs{\lambda}}/q_A^{\abs{\rho}})^{-s}.
    \end{equation}
    \label{lem:main}
\end{lemma}
\begin{proof}
    Using the definition of $\zeta_{\tl W}^A(s)$, we rewrite \eqref{eq:ext-fiber-moebius} as
    \begin{equation}
        \varepsilon_{\tl M}^{A\subeq B}(s)=\sum_{\tl W\subeq_B \tl M} \mu_B(\tl M/\tl W) \sum_{W\subeq_A \tl W} (\tl M:W)^{-s}.
    \end{equation}
    
    In the context of Lemma \ref{lem:ext-fiber-moebius}, let $\tl M=M_B(\lambda)$, and let $\mu,\nu$ be the type and cotype of $\tl W$ in $\tl M$ (over $B$) respectively, and let $\rho$ be the type of $W$ over $A$. By Lemma \ref{lem:moebius-vanishing}, it suffices to sum over $\nu$ of the form $(1^r)$ with $r\geq 0$. Putting everything together, we are done.
\end{proof}

It is sometimes necessary to consider product of finite DVR quotients. Suppose $B=\prod_{i=1}^l B_i$, where each $B_i$ is a finite DVR quotient with its own $q_{B,i}$ and $m_{B,i}$. Write $l=l(B)$. We treat $\underline{m}_B=(m_{B,i})_i$ and $\underline{q}_B=(q_{B,i})_i$ as $l$-tuples. A finite module over $B$ is of the form
\begin{equation}
    M_B(\underline{\lambda})=\bigoplus_{i=1}^l M_{B_i}(\lambda^{(i)}),
\end{equation}
where $\underline{\lambda}=(\lambda^{(i)})_i$ is an $l$-tuple of partitions such that $\lambda^{(i)}_1\leq m_{B,i}$; we write the condition as $\underline{\lambda}_1\leq \underline{m}_B$. The coset $\mathcal{P}_B(M_B(\underline{\lambda}))$ is isomorphic to the product poset $\prod_{i=1}^l \mathcal{P}_{B_i}(M_{B_i}(\lambda^{(i)}))$, so the machinery leading to \Cref{lem:main} generalizes in a straightforward way. To state the formula, let $\mathcal{P}$ denote the set of partitions, so $\mathcal{P}^l$ denotes the set of $l$-tuples of partitions. Given $\underline{\lambda},\underline{\mu},\underline{\nu}\in \mathcal{P}^l$, define
\begin{equation}
    g^{\underline{\lambda}}_{\underline{\mu}\underline{\nu}}(t_1,\dots,t_l)=\prod_{i=1}^l g^{\lambda^{(i)}}_{\mu^{(i)}\nu^{(i)}}(t_i), \quad g^{\underline{\lambda}}_{\underline{\mu}}(t_1,\dots,t_l)=\sum_{\underline{\nu}\in \mathcal{P}^l} g^{\underline{\lambda}}_{\underline{\mu}\underline{\nu}}(t_1,\dots,t_l),
\end{equation}
so it is clear that
\begin{equation}
    g^{\underline{\lambda}}_{\underline{\mu}}(t_1,\dots,t_l)=\prod_{i=1}^l g^{\lambda^{(i)}}_{\mu^{(i)}}(t_i).
\end{equation}
Finally, the notation $\underline{\lambda}\down^B_A$ is similarly defined for an $l(B)$-tuple $\underline{\lambda}$ of partitions with $\underline{\lambda}_1\leq \underline{m}_B$, and outputs an $l(A)$-tuple of partitions. Using the multi-index notation $(a_1,\dots,a_l)^{(b_1,\dots,b_l)}:=a_1^{b_1}\cdots a_l^{b_l}$, we have the basic relations
\begin{equation}
    \abs{M_B(\lambda)}=\underline{q}_B^{\abs{\underline{\lambda}}}=\underline{q}_A^{\abs{\underline{\lambda}\down^B_A}}.
\end{equation}

\begin{lemma}[Key lemma, multi-factor version]
    Let $A\subeq B$ be an inclusion of finite products of DVR quotients. Then for any $l(B)$-tuple $\underline{\lambda}$ of partitions with $\underline{\lambda}_1\leq \underline{m}_B$,
    \begin{equation}
        \varepsilon_{M_B(\underline{\lambda})}^{A\subeq B}(s)=\sum_{\substack{\mu^{(1)},\dots,\mu^{(l(B))}\\r_1,\dots,r_{l(B)}}} g^{\underline{\lambda}}_{\underline{\mu},(1^{\underline{r}})}(\underline{q}_B) (-1)^{\sum_i r_i} \prod_{i=1}^{l(B)}q_{B,i}^{\binom{r_i}{2}}  \sum_{\rho^{(1)},\dots,\rho^{(l(A))}} g^{\underline{\mu}\down_A^B}_{\underline\rho}(\underline q_A) (\underline q_B^{\abs{\underline \lambda}}/\underline q_A^{\abs{\underline \rho}})^{-s}.
    \end{equation}
    \label{lem:main-multi}
\end{lemma}

\subsection{Proof of \Cref{thm:rtilde}}

\begin{proof}
    The formulas in \Cref{thm:rtilde} follow from applying \Cref{lem:tlRn} and \Cref{lem:solomon}. The main task is to compute the saturation zeta function $\varepsilon_{B^n}^{A\subeq B}(s)$ for the specific conductor quotients $A \subeq B$ corresponding to each case. This is achieved by unpacking \Cref{lem:main} after setting $\lambda=(m_B^n)$ for a suitable $m_B$; note that $M_B(\lambda)\simeq B^n$. Let $m\in \Z_{\geq 1}$. \\
    
    \noindent \emph{Case 1 (ramified):} $R=\Fq[[X,Y]]/(Y^2-X^{2m+1})$. 
    
    We have $R\simeq \Fq[[T^2,T^{2m+1}]]\subeq \Fq[[T]]=\tl R$, $\fc_R=T^{2m}\tl R$, and
    \begin{equation}
        A=\frac{R}{\fc_R}=\frac{\Fq[[T^2]]}{(T^{2m})} \subeq \frac{\Fq[[T]]}{(T^{2m})} = \frac{\tl R}{\fc_R}=B.
    \end{equation}
    Thus $q_A=q_B=q$, $m_A=m$, $m_B=2m$, so that $B^n\simeq M_B(\lambda)$ for $\lambda=((2m)^n)$. To describe the rule to compute $\mu\down^B_A$, note that
    \begin{equation}
        \frac{\Fq[[T]]}{T^{k}}\simeq_{\Fq[[T^2]]} \frac{\Fq[[T^2]]}{(T^2)^{\ceil{k/2}}} \oplus \frac{\Fq[[T^2]]}{(T^2)^{\floor{k/2}}},
    \end{equation}
    so $\mu\down^B_A$ is the partition obtained by splitting each part of $\mu$ into two halves as evenly as possible. 

    Finally, since $\lambda=((2m)^n)$, by \Cref{lem:rectangular}, the polynomial $g^\lambda_{\mu,(1^r)}$ is nonzero only when $\mu=((2m)^n)-(1^r)$. \Cref{lem:main} thus reads
    \begin{equation}
        \varepsilon_{B^n}^{A\subeq B}(s)=\sum_{r} g^{((2m)^n)}_{(1^r)}(q) (-1)^r q^{\binom{r}{2}}  \sum_{\rho} g^{(((2m)^n)-(1^r))\down_A^B}_\rho(q) (q^{2mn}/q^{\abs{\rho}})^{-s}.
    \end{equation}
    By writing out each part of $((2m)^n)-(1^r)$ and applying the rule of $\mu\down^B_A$ described above, we have
    \begin{equation}
        (((2m)^n)-(1^r))\down_A^B = (m^{2n})-(1^r).
    \end{equation}
    At this point, the formula for $\varepsilon_{B^n}^{A\subeq B}(s)$ is fully explicit. Translating into the hypergeometric notation $G^\mathbf{r}_\mathbf{s}$ in \eqref{eq:hall_skew_hypergeom} and partially expanding out the formula, we get
    \begin{align}
        &\varepsilon_{B^n}^{A\subeq B}(s)=\sum_{r\in \Z, \mathbf{s}\in \Z^m} G_{(r)}^{(n^{2m})}(q) (-1)^r q^{\binom{r}{2}} G_{\mathbf s}^{((2n)^{m-1},2n-r)}(q) t^{2mn-\abs{\mathbf{s}}} \\
        &=\sum_{r\in \Z,\mathbf{s}\in \Z^m} \qbinom{n}{r}_q (-1)^r q^{\binom{r}{2}} t^{2mn-\abs{\mathbf{s}}} G^{((2n)^{m-1},2n-r)}_{\mathbf{s}}(q),\quad t:=q^{-s},
    \end{align}
    where as usual, $((2n)^{m-1},2n-r)$ means $(\underbrace{2n,2n,\dots,2n}_{m-1},2n-r)$. \\

    \noindent \emph{Case 2 (split):} $R=\Fq[[X,Y]]/(Y(Y-X^m)).$

    We use the model of $R$ in \cite[Eq.~(8.4)]{huangjiang2023torsionfree}. Consider the direct product ring $\Fq[[T_1]]\times \Fq[[T_2]]$, and write $e_1=(1,0), e_2=(0,1), T=(T_1,T_2)$. View $T_i:=e_i T$, $i=1,2$ also as an element of $\Fq[[T_1]]\times \Fq[[T_2]]$ by abuse of notation, so we have $T=T_1+T_2$. Then $R\incl \tl R$ can be identified with the inclusion
    \begin{equation}
        R=\Fq[[T_1^m,T_2^m,T]]\subeq \Fq[[T_1]]\times \Fq[[T_2]]=\tl R,
    \end{equation}
    with $\fc_R=(T_1^m,T_2^m)\tl R$, and
    \begin{equation}
        A=\frac{R}{\fc_R}=\frac{\Fq[[T]]}{(T^m)}\subeq \frac{\Fq[[T_1]]}{(T_1^m)}\times \frac{\Fq[[T_2]]}{(T_2^m)}=\frac{\tl R}{\fc_R}=B.
    \end{equation}
    Thus $q_A=q, m_A=m, \underline{q}_B=(q,q), \underline{m}_B=(m,m)$, so that $B^n=M_B(\underline{\lambda})$ for $\underline{\lambda}= ((m^n),(m^n))$. To compute the rule for $\underline{\mu}\down^B_A$, we note that
    \begin{equation}
        \frac{\Fq[[T_1]]}{(T_1^{k_1})}\times \frac{\Fq[[T_2]]}{(T_2^{k_2})}\simeq_{\Fq[[T]]} \frac{\Fq[[T]]}{(T^{k_1})}\oplus \frac{\Fq[[T]]}{(T^{k_2})},
    \end{equation}
    so $\underline{\mu}\down^B_A=\mu^{(1)}\sqcup \mu^{(2)}$ is a single partition obtained by collecting all parts of $\mu^{(1)}$ and $\mu^{(2)}$. 

    By \Cref{lem:main-multi}, setting $\mu^{(i)}=(m^n)-(1^{r_i})$ for $i=1,2$, and using the rule for $\underline{\mu}\down^B_A$ above, we get
    \begin{align}
        &\varepsilon^{A\subeq B}_{B^n}(s)=\sum_{r_1,r_2} g^{(m^n)}_{(1^{r_1})}(q)  g^{(m^n)}_{(1^{r_2})}(q) (-1)^{r_1+r_2}q^{\binom{r_1}{2}+\binom{r_2}{2}} \sum_{\rho} g^{(m^{2n-r_1-r_2},(m-1)^{r_1+r_2})}_\rho (q)(q^{2mn}/q^{\abs{\rho}})^{-s}\\
        &=\sum_{r_1,r_2\in \Z, \mathbf{s}\in \Z^m}\qbinom{n}{r_1}_q\qbinom{n}{r_2}_q(-1)^r q^{\binom{r_1}{2}+\binom{r_2}{2}}t^{2mn-\abs{\mathbf{s}}}G^{((2n)^{m-1},2n-r_1-r_2)}_\mathbf{s}(q),\quad t:=q^{-s}.
    \end{align}

    $\,$\\
    \noindent\emph{Case 3 (inert):} $R=\Fq[[T]]+T^m\F_{q^2}[[T]]$. 

    We have $R\subeq \F_{q^2}[[T]]=\tl R$, $\fc_R=T^m \tl R$, and
    \begin{equation}
        A=\frac{R}{\fc_R}=\frac{\Fq[[T]]}{(T^m)}\subeq \frac{\F_{q^2}[[T]]}{(T^m)}=\frac{\tl R}{\fc_R}=B.
    \end{equation}
    Thus $q_A=q$, $q_B=q^2$, $m_A=m_B=m$, so that $B^n=M_B(\lambda)$ for $\lambda=(m^n)$. To find the rule for $\mu\down^B_A$, note that
    \begin{equation}
        \frac{\F_{q^2}[[T]]}{(T^k)}\simeq_{\Fq[[T]]} \frac{\Fq[[T]]}{(T^k)}\oplus \frac{\Fq[[T]]}{(T^k)}
    \end{equation}
    (the isomorphism depends on a choice of ordered basis of $\F_{q^2}$ over $\Fq$), so $\mu\down^B_A=\mu\sqcup \mu$ obtained by duplicating each part of $\mu$. 

    As a result, \Cref{lem:main} reads
    \begin{align}
        &\varepsilon^{A\subeq B}_{B^n}(s)=\sum_r g^{(m^n)}_{(1^r)}(q^2) (-1)^r q^{2\binom{r}{2}}\sum_\rho g^{(m^{2n-2r},(m-1)^{2r})}_\rho (q) (q^{2mn}/q^{\abs{\rho}})^{-s} \\
        &=\sum_{r\in \Z, \mathbf{s}\in \Z^m} \qbinom{n}{r}_{q^2} (-1)^r q^{r^2-r} t^{2mn-\abs{\mathbf{s}}}G^{((2n)^{m-1},2n-2r)}_{\mathbf{s}}(q), \quad t:=q^{-s}.\label{eq:inert-saturation}
    \end{align}
\end{proof}

\section{Finitized Coh zeta function at $s=0$}\label{sec:bounded}

In this section, we compute the $s=0$ (or $t=1$) specialization of the normalized finitized Coh zeta function, $\nuhat_{R,n}(s)$. The results are a direct application of the key technical formulas from \Cref{thm:rtilde} combined with the fundamental evaluation formula \eqref{eq:nuhat-at-zero} from the preliminaries.

\begin{theorem}\label{thm:s-zero-evals}
For any $m\in \Z_{\geq 1}$ and $n\in \Z_{\geq 0}$, the normalized finitized Coh zeta functions at $s=0$ for the three families of quadratic orders are given by:
    \begin{enumerate}
        \item (Ramified) For $R=R_{2,2m+1}$,
        \[ \nuhat_{R,n}(0)=q^{-mn^2}\sum_{r\in \Z,\mathbf{s}\in \Z^m} \qbinom{n}{r}_q (-1)^r q^{\binom{r}{2}} G^{((2n)^{m-1},2n-r)}_{\mathbf{s}}(q). \]
        \item (Split) For $R=R_{2,2m}$,
        \[ \nuhat_{R,n}(0)=q^{-mn^2}\sum_{r_1,r_2\in \Z, \mathbf{s}\in \Z^m}\qbinom{n}{r_1}_q\qbinom{n}{r_2}_q(-1)^{r_1+r_2} q^{\binom{r_1}{2}+\binom{r_2}{2}}G^{((2n)^{m-1},2n-r_1-r_2)}_\mathbf{s}(q). \]
        \item (Inert) For $R=R'_{2,2m}$,
        \[ \nuhat_{R,n}(0)=q^{-mn^2}\sum_{r\in\Z, \mathbf{s}\in \Z^m} \qbinom{n}{r}_{q^2}(-1)^{r}q^{r^{2}-r}G_{\mathbf{s}}^{((2n)^{m-1},2n-2r)}(q). \]
    \end{enumerate}
\end{theorem}

\begin{proof}
Each of the three formulas is a direct consequence of the evaluation identity \eqref{eq:nuhat-at-zero}: $\nuhat_{R,n}(0) = D^{-n^2}\nu_{\tl R^n}^R(0)$. For all three families of quadratic orders, we have $D=\abs{\tl R/R}=q^m$. The term $\nu_{\tl R^n}^R(0) = \zeta_{\tl R^n}^R(0) / \zeta_{\tl R^n}^{\tl R}(0)$ is computed by specializing the corresponding formula for $\zeta_{\tl R^n}^R(s)$ from \Cref{thm:rtilde} at $s=0$ (so $t=1$) and dividing by the specialization of \Cref{lem:solomon}. Substituting these components into the evaluation identity yields the results.
\end{proof}

We now rewrite these formulas into a form that makes it clear that they converge coefficient-wise as $n\to \infty$. The transformation involves expanding the $G$-polynomials, converting all $q$-binomial coefficients to a common base of $q^{-1}$ (or $q^{-2}$), and applying the change of variables $t_i = n - s_{m+1-i}$. After significant simplification, including the telescoping of a product of $q$-binomials, we arrive at the following equivalent expressions.

\begin{proposition}\label{prop:s-zero-alternative}
The formulas in \Cref{thm:s-zero-evals} are equivalent to the following:
\begin{enumerate}
    \item (Ramified)
    \begin{equation}
    \begin{multlined}
        \nuhat_{R_{2,2m+1},n}(0) = \sum_{r\in\Z, \mathbf{s}\in\Z^m} \frac{(-1)^r q^{ - \left( \binom{s_1}{2} + \binom{r-s_1}{2} + r \right) - \sum_{k=2}^m s_k^2}}{\prod_{k=1}^{m-1} (q^{-1};q^{-1})_{s_k-s_{k+1}}} \\
        \times \qbinom{n}{r}_{q^{-1}} \qbinom{2n-r}{n-r+s_1}_{q^{-1}} \frac{(q^{-1};q^{-1})_{n+s_1}}{(q^{-1};q^{-1})_{n+s_m}}.
    \end{multlined}
    \end{equation}
    \item (Split)
    \begin{equation}
    \begin{multlined}
        \nuhat_{R_{2,2m},n}(0) = \sum_{r_1,r_2\in\Z, \mathbf{s}\in\Z^m} \frac{(-1)^{r_1+r_2} q^{ -\left(\binom{s_1}{2} + \binom{r_1+r_2-s_1}{2} + r_1+r_2\right) - \binom{r_1+1}{2} - \binom{r_2+1}{2} - \sum_{k=2}^m s_k^2}}{\prod_{k=1}^{m-1} (q^{-1};q^{-1})_{s_k-s_{k+1}}} \\
        \times \qbinom{n}{r_1}_{q^{-1}}\qbinom{n}{r_2}_{q^{-1}} \qbinom{2n-r_1-r_2}{n-r_1-r_2+s_1}_{q^{-1}} \frac{(q^{-1};q^{-1})_{n+s_1}}{(q^{-1};q^{-1})_{n+s_m}}.
    \end{multlined}
    \end{equation}
    \item (Inert)
    \begin{equation}
    \begin{multlined}
        \nuhat_{R'_{2,2m},n}(0) = \sum_{r\in\Z, \mathbf{s}\in\Z^m} \frac{(-1)^r q^{-r-(s_1-r)^2-\sum_{k=2}^m s_k^2}}{\prod_{k=1}^{m-1} (q^{-1};q^{-1})_{s_k-s_{k+1}}} \\
        \times \qbinom{n}{r}_{q^{-2}} \qbinom{2n-2r}{n-s_1}_{q^{-1}} \frac{(q^{-1};q^{-1})_{n+s_1}}{(q^{-1};q^{-1})_{n+s_m}}.
    \end{multlined}
    \end{equation}
\end{enumerate}
\end{proposition}

These explicit formulas allow for a direct comparison with the $t=1$ specializations of the finitized Coh zeta functions from the introduction. To do so, we first express the results of \Cref{thm:known} in terms of the normalized zeta function $\nuhat_{R,n}(s)$. For the ramified case, $\tl R = \Fq[[T]]$, so $\zetahat_{\tl R,n}(s) = (tq^{-1};q^{-1})_n^{-1}$. For the split case, $\tl R = \Fq[[T_1]] \times \Fq[[T_2]]$, so $\zetahat_{\tl R,n}(s) = (tq^{-1};q^{-1})_n^{-2}$. This gives:
\begin{align*}
    \nuhat_{R_{2,2m+1},n}(s) &= \AG_n(q^{-1},t;2m+3) \\
    \nuhat_{R_{2,2m},n}(s) &= (tq^{-1};q^{-1})_n \Br_n(q^{-1},-t;2m+2).
\end{align*}
Specializing at $s=0$ ($t=1$) and using the known identity $\Br_n(q,-1;2m+2) = 1/(q;q)_n$, we find that
\begin{align}
    \nuhat_{R_{2,2m+1},n}(0) &= \AG_n(q^{-1},1;2m+3) \label{eq:nuhat-ramified-known} \\
    \nuhat_{R_{2,2m},n}(0) &= 1. \label{eq:nuhat-split-known}
\end{align}
Equating these known expressions with the formulas in \Cref{prop:s-zero-alternative} (a) and (b) yields two complex multi-sum $q$-series identities. Since both sides of each identity are derived from the same zeta function, these identities are now theorems with ``$q$-combinatorial'' proofs.

For the inert case, \Cref{conj:main} predicts what $\zetahat_{R'_{2,2m},n}(0)$ should be. To compare, we must normalize correctly using $\zetahat_{\tl R,n}(0) = (q^{-2};q^{-2})_n^{-1}$. This yields the prediction:
\begin{equation} \label{eq:nuhat-inert-conj}
    \nuhat_{R'_{2,2m},n}(0) \overset{?}{=} (-q^{-1};q^{-1})_n \Br_n(q^{-1},1;2m+2).
\end{equation}
Therefore, the $s=0$ assertion of our main conjecture is equivalent to the $q$-series identity obtained by equating the formula in \Cref{prop:s-zero-alternative} (c) with the right-hand side of \eqref{eq:nuhat-inert-conj}. 

\section{The case $m=1$, inert}\label{sec:m=1}

In this section, we compute an explicit $q$-series formula for the finitized Coh zeta function of the simplest inert order, $R=R'_{2,2}$. This computation achieves a strong form of verification for the $m=1$ case of our main conjecture. We reduce the conjecture to a new, explicit $q$-series identity which, being a multi-sum identity with a fixed number of summation variables, is algorithmically provable by the $q$-Wilf-Zeilberger method \cite{PWZ96}.

\subsection{A simpler formula for $\zeta_{\tl R^n}^R(s)$}
For $m=1$, the ring is $R = \Fq[[T]]+T\Fqsq[[T]]$. The conductor is $\fc_R=T\tl R$, so the quotient rings in \Cref{lem:tlRn} are simply $A=\Fq$ and $B=\Fqsq$. The saturation zeta function $\varepsilon_{B^n}^{A\subeq B}(s)$ in this case can be computed by a direct combinatorial counting argument, which provides a simpler alternative to the general formula in \Cref{thm:rtilde}(c). This counting problem is equivalent to enumerating $\Fq$-subspaces of an $\Fqsq$-vector space with a particular spanning property.

\begin{proposition}\label{prop:m=1-rtilde-alt}
~
\begin{enumerate}
    \item We have the identity in $\Z[q]$:
    \begin{equation}
        q^{\binom{r}{2}}\frac{(q^{2n};q^{-2})_r}{(q;q)_r} = \sum_{i=0}^n  (-1)^i q^{i^2-i}\qbinom{n}{i}_{q^2} \qbinom{2n-2i}{2n-r}_q.
    \end{equation}
    When $q$ is a prime power, both sides count the number of $r$-codimensional $\Fq$-vector subspaces $W$ of $\F_{q^2}^{n}$ satisfying $\F_{q^2}\cdot W = \F_{q^2}^{n}$.
    \item Let $R=R'_{2,2}$. The lattice zeta function $\zeta_{\tl R^n}^R(s)$ has the alternative formula:
    \begin{equation}\label{eq:inert-rtilde-m=1}
        (q^{-2s};q^2)_n \zeta_{\tl R^{n}}^R(s) = \sum_{r=0}^n q^{\binom{r}{2}}\frac{(q^{2n};q^{-2})_r}{(q;q)_r} t^r, \quad t:=q^{-s}.
    \end{equation}
\end{enumerate}
\end{proposition}

\begin{proof}
Part (b) follows from part (a) and the definition of the saturation zeta function. Part (a) is a combinatorial identity proven by counting the same set of vector subspaces in two ways. The right-hand side is the formula derived from the M\"obius inversion machinery (the $m=1$ case of the proof of \Cref{thm:rtilde}(c)). The left-hand side is derived by a counting argument using dual vector spaces, which we outline below.
\end{proof}

\subsubsection{Proof of \Cref{prop:m=1-rtilde-alt}(a) via dual spaces}
Let $k=\Fq$ and $l=\Fqsq$. For a finite-dimensional $l$-vector space $V$, we define the dual space $V^\vee:=\Hom_k(V,k)$, which is itself an $l$-vector space via $(c\cdot\alpha)(v):=\alpha(cv)$ for $c\in l$. For a $k$-subspace $W\subeq V$, we define $W^\perp:=\set{\alpha\in V^\vee: \alpha(W)=0}$. The map $W \mapsto W^\perp$ is a dimension-reversing bijection between $k$-subspaces of $V$ and $V^\vee$.

A key fact for the quadratic extension $l/k$ is that for any $k$-subspace $W \subseteq V$, the dimensions of its $l$-closure $\bbar{W}:=lW$ and its $l$-interior (the largest $l$-subspace contained in $W$, denoted by $\underline{W}$) are related by the following.

\begin{lemma}\label{lem:dimension-clos-int}
    Let $V$ be a finite-dimensional $l$-vector space, and $W$ be a $k$-subspace of $V$. Then
    \begin{equation}
        \dim_l \bbar{W}+\dim_l \underline{W}=\dim_k W.
    \end{equation}
\end{lemma}
\begin{proof}
    Since $l/k$ is a quadratic extension, we observe that for any element $c\in l\setminus k$, we have
    \begin{equation}
        \bbar{W}=W+cW, \quad \underline{W}=W\cap cW.
    \end{equation}
    For these $k$-vector spaces, we apply a general dimension identity:
    \begin{equation}
        \dim_k (W+cW) + \dim_k (W\cap cW)= \dim_k W+\dim_k cW=2\dim_k W.
    \end{equation}
    Halving both sides yields the conclusion of the lemma.
\end{proof}

\begin{proof}[Proof of \Cref{prop:m=1-rtilde-alt}(a)]
The relation in \Cref{lem:dimension-clos-int}, along with the property $(\bbar{W})^\perp = \underline{(W^\perp)}$, implies that a $k$-subspace $W\subeq V$ of dimension $2n-r$ satisfies $\bbar{W}=V$ if and only if its dual $W^\perp \subeq V^\vee$ is an $r$-dimensional $k$-subspace whose $l$-closure $\bbar{W^\perp}$ is also $r$-dimensional over $l$.

The counting problem thus reduces to counting the number of $r$-dimensional $k$-subspaces of $V^\vee \simeq l^n$ whose basis vectors are linearly independent over $l$. The number of ways to choose an ordered $l$-linearly independent set of $r$ vectors in $l^n$ is $(q^{2n}-1)(q^{2n}-q^2)\cdots(q^{2n}-q^{2(r-1)})$. Dividing by the number of ways to choose an ordered $k$-basis for an $r$-dimensional $k$-subspace, which is $(q^r-1)(q^r-q)\cdots(q^r-q^{r-1})$, yields the formula $q^{\binom{r}{2}}\frac{(q^{2n};q^{-2})_r}{(q;q)_r}$.
\end{proof}

\subsection{The finitized Coh zeta function}
We can now compute the full finitized Coh zeta function for $R'_{2,2}$, incorporating the simplified formula for $\zeta_{\tl R^n}^R(s)$ along the way.

\begin{theorem}\label{thm:m=1-main-result}
    For $R=R'_{2,2}$, the finitized Coh zeta function is given by the double sum:
    \begin{equation}\label{eq:inert-m=1-bounded-zeta}
        \zetahat_{R,n}(s)=\sum_{i\geq j\geq 0} (-1)^j q^{-(i^2+ij+j)} \qbinom{n}{i}_{q^{-1}}\frac{(q^{2i};q^{-2})_j }{(q^{-1};q^{-1})_j (t^2 q^{-2};q^{-2})_i} t^{i+j}.
    \end{equation}
\end{theorem}

\begin{proof}
A key observation for $m=1$ is that the maximal ideal $\m = T\Fqsq[[T]]$ of $R$ is isomorphic to the normalization $\tl R = \Fqsq[[T]]$ as an $R$-module. Thus, $\zeta_{\m R^r}^R(s) = \zeta_{\tl R^r}^R(s)$ for any $r$. We apply \Cref{lem:nakayama}:
\[ \zeta_{R^n}^R(s)=\sum_{i=0}^n \qbinom{n}{i}_q q^{-is}\zeta_{\m R^i}^R(s-n+i) = \sum_{i=0}^n \qbinom{n}{i}_q q^{-is}\zeta_{\tl R^i}^R(s-n+i). \]
Substituting the simplified formula for $\zeta_{\tl R^i}^R(s-n+i)$ from \Cref{prop:m=1-rtilde-alt}(b) gives an explicit expression for $\zeta_{R^n}^R(s)$. The formula for $\zetahat_{R,n}(s) = \zeta_{R^n}^R(s+n)$ is then obtained by substituting $s \mapsto s+n$ and simplifying.
\end{proof}

We now connect this result to our main conjecture. \Cref{conj:main} for $m=1$ predicts that $\zetahat_{R'_{2,2},n}(s)$ is given by a finitization of a Bressoud sum, which is a single summation:
\begin{equation}
\zetahat_{R'_{2,2},n}(s) \overset{?}{=} \frac{1}{(tq^{-1};q^{-1})_n} \Br_n(q^{-1},t;4).
\end{equation}
\Cref{thm:m=1-main-result} provides the first explicit $q$-series formula for the left-hand side---a double sum---derived from algebraic principles. Therefore, the $m=1$ case of our main conjecture is equivalent to a new $q$-series identity equating a double sum to a single sum. Crucially, an identity of this type is algorithmically verifiable using the $q$-Wilf-Zeilberger (WZ) method \cite{PWZ96}. While we do not carry out the WZ proof in this paper, this reduction shows that the $m=1$ case of the conjecture is, in principle, provable by established techniques.

\section*{Acknowledgements}

The author thanks Shane Chern and S.~Ole Warnaar for helpful conversations. The author is extremely grateful to the organizers and participants of the Legacy of Ramanujan 2024 conference; a poster presentation there created the opportunity to pose the problems about the mysterious split case in \cite{huangjiang2023torsionfree} to a wide audience, among which Shane Chern discovered the crucial simplification \cite{shane2024multiple}. The author acknowledges the use of the AI model Gemini 2.5 Pro in improving the exposition and organization of this paper. All mathematical ideas and results presented herein are the author's own.

\end{document}